\UseRawInputEncoding
\documentclass[12pt, reqno]{amsart}
\usepackage[margin=1in]{geometry}
\usepackage[mathscr]{eucal}
\usepackage{bm}
\usepackage{mathptmx}
\usepackage{amssymb}
\usepackage{amsthm}
\usepackage{dcolumn}
\usepackage[all]{xy}
\usepackage{enumitem}
\usepackage{xcolor}
\usepackage[hidelinks]{hyperref}

\def\textmatrix#1&#2\\#3&#4\\{\bigl({#1 \atop #3}\ {#2 \atop #4}\bigr)}
\def\dispmatrix#1&#2\\#3&#4\\{\left({#1 \atop #3}\ {#2 \atop #4}\right)}

\newcommand{\beg}{\begin{equation}}
        \newcommand{\eeg}{\end{equation}}
\newcommand{\ben}{\begin{eqnarray*}}
        \newcommand{\een}{\end{eqnarray*}}

\newtheorem{thm}{Theorem}[section]

\newtheorem{lem}[thm]{Lemma}

\numberwithin{equation}{section} \theoremstyle{definition}
\newtheorem{defn}[thm]{Definition}
\newtheorem{rem}[thm]{Remark}

\newtheorem{eg}[thm]{Example}
\def\textmatrix#1&#2\\#3&#4\\{\bigl({#1 \atop #3}\ {#2 \atop #4}\bigr)}
\def\dispmatrix#1&#2\\#3&#4\\{\left({#1 \atop #3}\ {#2 \atop #4}\right)}
\newcommand{\C}{\mathbb{C}}
\newcommand{\D}{\mathbb{D}}
\newcommand{\N}{\mathbb{N}}
\newcommand{\R}{\mathbb{R}}

\newcommand{\X}{\mathbb{X}}
\newcommand{\Y}{\mathbb{Y}}

\newcommand{\HS}{\mathcal{H}}
\newcommand{\KS}{\mathcal{K}}

\title[On a minimal And\^{o} dilation]{On a minimal And\^{o} dilation for a pair of strict contractions}

\author[Jana and Pal]{Swapan Jana and Sourav Pal}

\address[Swapan Jana]{Mathematics Department, Indian Institute of Technology Bombay,
		Powai, Mumbai - 400076, India.} \email{ swapan.jana@iitb.ac.in , swapan.math2015@gmail.com}

\address[Sourav Pal]{Mathematics Department, Indian Institute of Technology Bombay,
		Powai, Mumbai - 400076, India.} \email{souravpal@iitb.ac.in , souravmaths@gmail.com}

\keywords{Strict contraction, Banach space contraction, And\^{o} dilation}	
	
\subjclass[2020]{47A20, 47B01, 46B03}

\thanks{The first named author is supported by the ``Prime Minister's Research Fellowship (PMRF)" with Award No. PMRF-1302045. The second named author is supported in part by Core Research Grant with Award No. CRG/2023/005223 from Anusandhan National Research Foundation (ANRF) of Govt. of India.}

\begin{document}

\begin{abstract}

The isometric dilation of a pair of commuting contractions due to And\^{o} is not minimal. We modify And\^{o}'s dilation and construct a minimal isometric dilation on $\mathcal H \oplus_2 \ell_2(\mathcal H \oplus_2 \mathcal H)$ for a commuting pair of strict contractions on a Hilbert space $\mathcal H$. In the same spirit, we construct under certain conditions a minimal And\^{o} dilation for a commuting pair of strict Banach space contractions. Further, we show that an And\^{o} dilation is possible even for a more general pair of commuting contractions $(T_1,T_2)$ on a normed space $\mathbb X$ provided that the function $A_{T_i}: \mathbb X \rightarrow \mathbb R$ given by $A_{T_i}(x)=(\|x\|^2-\|T_ix\|^2)^{\frac{1}{2}}$ defines a norm on $\mathbb X$ for $i=1,2$.

\end{abstract}

\maketitle

\section{Introduction}\label{Sec:01}

\noindent
Throughout the paper, all Banach and Hilbert spaces are over the field of complex numbers $\C$ and all operators are bounded linear operators. For two normed linear spaces $\X_1$ and $\X_2$, their \textit{direct $2$-sum} (or, \textit{orthogonal $2$-sum}) $\X_1 \oplus_2 \X_2$ is the normed linear space consisting of the pairs $(x_1,x_2)$ with $x_i\in \X_i$ ($i=1,2$) and is equipped with the norm $\|(x_1,x_2)\| = (\|x_1\|^2 + \|x_2\|^2)^{\frac{1}{2}}$. To maintain uniformity, the direct sum of two Hilbert spaces $\HS_1$ and $\HS_2$ (which is also orthogonal $2$-sum in case of Hilbert spaces) will also be denoted by $\HS_1 \oplus_2 \HS_2$. For any Banach space $\X$, the space $\ell_2(\X)$ consists of all square-summable sequences $\{x_n\}$ with entries in $\X$ and is equipped with the norm $\|\{x_n\}\|=(\sum_{n=1}^\infty \|x_n\|^2)^{\frac{1}{2}}$. An operator $T$ on a Banach space $\X$ is said to be a \textit{contraction} if $\|T\|\leq 1$ and it is called a \textit{strict contraction} if $\|T\|< 1$. For a Hilbert space contraction $T$, the defect operator and the defect space of $T$ are $D_T= (I -T^*T)^{\frac{1}{2}}$ and $\mathfrak{D}_T=\overline{Ran } \ D_T$, respectively.

\medskip

A commuting tuple of contractions $(T_1, \dots , T_n)$ acting on a Hilbert space $\HS$ is said to admit an \textit{isometric dilation} if there is a Hilbert space $\mathcal K$ that contains $\HS$ as a closed linear subspace, i.e., $\mathcal K= \HS \oplus_2 \HS_1$ for some Hilbert space $\HS_1$ and a tuple of commuting isometries $(V_1, \dots , V_n)$ acting on $\mathcal K$ such that
\begin{equation}\label{eq:101}
    P_{_{\HS}} V_1^{s_1}\dots V_n^{s_n} h= T_1^{s_1}\dots T_n^{s_n}h, \quad h\in \HS, \quad s_1, \dots ,s_n \in \N \cup\{0\},
\end{equation}
where $P_{_{\HS}}: \mathcal K \rightarrow \HS$ is the orthogonal projection. Such a dilation is said to be \textit{minimal} if 
\[
  \KS = \overline{span}\left\{V_1^{s_1}\cdots V_n^{s_n}h : h\in \HS, ~ s_1, \dotsc s_n\in \N \cup \{0\}\right\}.
\]
In 1953, B\'{e}la Sz.-Nagy in his famous paper \cite{BN} proved that every Hilbert space contraction dilates to an isometry. Moreover, minimal isometric dilation of a contraction is unique up to unitary equivalence, that is, if $V_1$ on $\mathcal K_1$ and $V_2$ on $\mathcal K_2$ are two minimal isometric dilations of a contraction $T$ on $\HS$, then there is a unitary $W:\mathcal K_1 \rightarrow \mathcal K_2$ such that $V_2=WV_1W^{-1}$. Also, up to unitary equivalence the minimal isometric dilation space of a contraction $T$ on $\HS$ is uniquely determined to be $\HS \oplus_2 \ell_2(\mathfrak{D}_T)$, e.g., see Sch\"{a}ffer's explicit isometric dilation in \cite{JJS}. Note that, $\mathfrak{D}_T=\HS$ if $T$ is a strict contraction and consequently the minimal dilation space becomes $\HS \oplus_2 \ell_2(\HS)$. Later, in 1963 And\^{o} \cite{TA} showed by an explicit construction that an isometric dilation is also possible for a pair of commuting Hilbert space contractions. However, Parrott \cite{SP} established by a counterexample that such an isometric dilation is not always possible for three or more commuting contractions. Two decades ago, in the seminal paper \cite{AMc}, Agler and McCarthy had sharpen And\^{o} dilation of a pair of commuting contractive matrices without unimodular eigenvalues on a distinguished variey in the bidisc.   Apart from these, there is an extensive literature on dilation of Hilbert space contractions in one and several variables, e.g., see \cite{AMc, B-M-S, B-Su-1, BLTT, BL, CW, CV I, CV II, DS-01, LT, VM, MV, Timotin} and the references therein.

\smallskip

 In contrast to the one variable case, a minimal isometric dilation in more than one variable (if exists) is not unique upto unitary equivalence, see \cite{LT}. In Theorem \ref{thm:201} of this paper, we construct a minimal isometric dilation on $\HS \oplus_2 \ell_2(\HS \oplus_2 \HS)$ for a pair of commuting strict contractions $(T_1,T_2)$ acting on a Hilbert space $\HS$. Note that, And\^{o} considered the space $\HS \oplus_2 \ell_2(\HS \oplus_2 \HS \oplus_2 \HS \oplus_2 \HS)$ and a pair of commuting isometries acting on it for his dilation, but the dilation was not minimal. Our dilation here is minimal and it reveals a similarity in the pattern of structures of minimal dilation spaces in the following way: in one-variable theory the minimal isometric dilation space for any strict contraction on $\HS$ is $\HS \oplus_2 \ell_2(\HS)$, and for any pair of commuting strict contractions on $\HS$ a minimal dilation space could be $\HS \oplus_2 \ell_2(\HS \oplus_2 \HS)$. Nevertheless, our construction of dilation is based on And\^{o}'s original idea (as in \cite{TA}).

\smallskip

Taking cue from the explicit Hilbert space dilation of Theorem \ref{thm:201}, next we construct an isometric dilation under some conditions for a commuting pair of strict contractions on a Banach space $\X$. Finding an explicit dilation in Banach space setting is not an easy task, mainly because of the absence of adjoint of an operator. We bypass this issue in a certain way. Before coming to this point, we must mention that the definition of Banach space dilation is more subtle. The primary reason is that a Banach space projection can have norm strictly greater than $1$. However, if the dilation space is equal (or isometric) to $\X \oplus_2 \mathbb L$ for some Banach space $\mathbb L$, which goes parallel with the definition of Hilbert space dilation, then the projection onto $\X$ has norm equal to $1$. It is easy to verify this statement. Indeed, if $P$ is the projection of $\X \oplus_2 \mathbb{L}$ onto $\X$, then for any $(x, l)\in \X \oplus_2 \mathbb{L}$, we have $\|P(x, l)\| = \|x\| \leq \left( \|x\|^2+\|l\|^2 \right)^{\frac{1}{2}}=\|(x,l)\|$. So, $\|P\| \leq 1$. Also, $P^2 = P$ implies $1 \leq \|P\|$. Thus, $\|P\|=1$. Also, one can see \cite{Franchetti & Cheney, Morris & Cheney} for a detailed discussion on norms of Banach projections. Consequently, one can define isometric dilation for Banach space contractions in the following way that generalizes the Hilbert space dilation.

\begin{defn}
A commuting tuple of contractions $(T_1,\dots ,T_n)$ acting on a Banach space $\X$ is said to have an \textit{isometric dilation} if there is a Banach space $\widetilde{\X}$ of the form $\widetilde{\X}=\X \oplus_2 \mathbb L$ for some Banach space $\mathbb L$ and a tuple of commuting isometries $(V_1, \dots ,V_n)$ on $\widetilde{\X}$ such that
\begin{equation}\label{eq:103}
  P_{_{\X}}V_1^{s_1}\dots V_n^{s_n}~x = {T}_1^{~s_1}\dots {T}_n^{~s_n} ~x,\quad x\in \X, \quad s_1,\dots ,s_n \in \N \cup \{0\},
\end{equation}
where $P_{_{\X}}: \widetilde{\X}\rightarrow \X$ is the orthogonal projection. Moreover, such a dilation is called \textit{minimal} if 
\[
 \widetilde{\X} = \overline{span}\left\{V_1^{s_1}\dotsc V_n^{s_n}x : x\in \X, ~ s_1, \dotsc, s_n\in \N \cup \{0\}\right\}.
\]
\end{defn}

There is a considerable amount of literature for Banach space dilation in one variable, e.g., see \cite{Stroescu, AL 2, SJ, JP-1, JPR} and the references therein. In this context, let us mention that the following theorem due to the authors and Roy \cite{JPR} characterizes all strict Banach space contractions that dilate to isometries.

\begin{thm}[\cite{JPR}, Theorem 7.13]\label{thm:101}
A strict contraction $T$ on a Banach space $\X$ dilates to a Banach space isometry if and only if the map $A_T:\X \to [0,\infty)$ defined by
\begin{equation}\label{eqn:new-Sec1-001}
  A_T(x) = \left(\|x\|^2 - \|Tx\|^2\right)^{\frac{1}{2}}, \quad x\in \X
\end{equation}
induces a norm on $\X$. Moreover, the corresponding minimal isometric dilation space of $T$ is isometrically isomorphic to $\mathbb{X}\oplus_2 \ell_2(\mathbb{X}_0)$, where $\mathbb{X}_0$ is the Banach space $(\mathbb{X}, A_T)$.
\end{thm}

If a commuting pair of strict Banach space contractions $(T_1,T_2)$ possesses an isometric dilation $(V_1,V_2)$, then each $T_i$ admits isometric dilation to $V_i$ and hence by Theorem \ref{thm:101} the map $A_{T_i}$ defines a norm on $\X$ for $i=1,2$. Thus, the fact that each $A_{T_i}$ defines a norm on $\X$ is a necessary condition for dilation of a strict Banach space pair $(T_1,T_2)$. We shall need one more condition to achieve our desired Banach space isometric dilation, and that condition will be clarified once we learn the proof of Theorem \ref{thm:201}. We obtain this dilation in Theorem \ref{thm:202}, another main result of this article. Finally, in Theorem \ref{thm:301} we show that our mechanism of constructing Banach space isometric dilation (as in Theorem \ref{thm:202}) is effective even for a broader class of contractions on normed spaces.

\section{Dilation for a pair of strict contractions }\label{sec:02}

\noindent
 It is evident from And\^{o}'s construction of isometric dilation \cite{TA} that the idea of Sch\"{a}ffer's explicit dilation in one variable \cite{JJS} has been modified as a first step. Indeed, for a commuting pair of contractions $(T_1,T_2)$ acting on a Hilbert space $\HS$, And\^{o} first considered the Hilbert space $\widehat{\KS} = \HS \oplus_2 \ell_2(\HS)$, and the isometries $V_i: \widehat{\KS} \to \widehat{\KS}$, defined by 
\[
  V_i(h, h_1, h_2, \dotsc) = (T_ih, D_{T_i}h, \mathbf{0}, h_1, h_2, \dotsc), \quad (h, h_1, h_2, \dotsc)\in \widehat{\KS}, \quad i=1,2.
\]
However, these isometries failed to dilate $(T_1,T_2)$ due to lack of commutativity. Then, he considered the Hilbert space $\KS = \HS \oplus_2 \ell_2(\HS^4)$, where $\HS^4=\HS \oplus_2 \HS \oplus_2 \HS \oplus_2 \HS$ and set the isometries $V_1$, $V_2$ on $\KS$ accordingly via a unitary $U: \HS^4 \to \HS^4$ satisfying
\begin{equation}\label{eq:102}
  U(D_{T_1}T_2h, \mathbf{0}, D_{T_2}h, \mathbf{0}) = (D_{T_2}T_1h, \mathbf{0}, D_{T_1}h, \mathbf{0}), \quad h\in \HS,
\end{equation}
which finally established the identity \eqref{eq:101} for $n=2$. Here, we replace the space $\HS^4$ by $\HS^2=\HS \oplus_2 \HS$ when $T_1,T_2$ are strict contractions. More precisely, we show that if $\|T_1\|, \|T_2\| <1$, then the alternate zeros in (\ref{eq:102}) can be removed and the desired isometries $V_1,V_2$ can be defined on $\HS \oplus_2 \ell_2(\HS \oplus_2 \HS)$. Moreover, the dilation $(V_1, V_2)$ on $\HS \oplus_2 \ell_2(\HS \oplus_2 \HS)$ is minimal.

\begin{thm}\label{thm:201}
A commuting pair of strict contractions on a Hilbert space $\HS$ admits a minimal isometric dilation on $\HS\oplus_2 \ell_2(\HS \oplus_2 \HS)$.
\end{thm}

\begin{proof}
Let $(T_1, T_2)$ be a commuting pair of strict contractions on $\HS$, i.e., $\|T_i\|< 1$ for $i=1,2$. Let $T= T_1T_2$ and consider the subspaces
\begin{equation} \label{eqn:new-021}
M_1= \{(D_{T_1}h, D_{T_2}T_1h): h\in \HS\}, \quad \& \quad M_2= \{(D_{T_1}T_2h, D_{T_2}h): h\in \HS\}.
\end{equation}
First we show that the linear maps $W_i: \mathfrak{D}_{T} \to M_i$ ($i=1,2$) defined by $W_1(D_{T}h)= (D_{T_1}h, D_{T_2}T_1h)$ and $W_2(D_{T}h)= (D_{T_1}T_2h, D_{T_2}h)$, respectively are isometries. For any $h\in \HS$,
\begin{align*}
  \|W_1(D_{T}h)\|= \|(D_{T_1}h, D_{T_2}T_1h)\| & =\left(\|D_{T_1}h\|^2 + \|D_{T_2}T_1h\|^2 \right)^{\frac{1}{2}}\\
  & = \left(\|h\|^2 - \|T_1h\|^2 + \|T_1h\|^2 - \|T_1T_2h\|^2\right)^{\frac{1}{2}} = \|D_Th\|,
\end{align*}
and similarly, $\|W_2(D_{T}h)\|= \|(D_{T_1}T_2h, D_{T_2}h)\| = \|D_Th\|$.
Since $W_1,W_2$ are isometries, it is evident that the subspaces $M_1$ and $M_2$ are closed in $\HS\oplus_2\HS$ and they are isometrically isomorphic to $\mathfrak{D}_T$ via $W_1$ and $W_2$, respectively. Thus, $\dim(M_i)= \dim(\mathfrak{D}_T)$ for $i=1,2$ and consequently, the map $Q: M_2 \to M_1 $ defined by
\begin{equation} \label{eqn:201-A}
  Q(D_{T_1}T_2h, D_{T_2}h) = (D_{T_1}h, D_{T_2}T_1h), \quad \quad h\in \HS,
\end{equation}
is a unitary. Now, we show that $M_1 \cap M_2 = \{\mathbf{0}\}$. Let $z\in M_1 \cap M_2$. Then there exist $h_1,~h_2\in \HS$ such that
\begin{equation}\label{eq:201}
  (D_{T_1}h_1, D_{T_2}T_1h_1) = z = (D_{T_1}T_2h_2, D_{T_2}h_2).
\end{equation}
Since $\|T_i\| < 1$ as $T_i$ is a strict contraction, $D_{T_i}$ is invertible for $i=1,2$. The component wise equality in \eqref{eq:201} and the injectivity of $D_{T_i}$ imply that $h_1= T_2h_2$ and $T_1h_1= h_2$. So, we have $T_1T_2h_i = Th_i= h_i$ for $i=1,2$, which is possible only for $h_1=\mathbf{0}=h_2$, since $\|T\|<1$. Consequently, we have $z=\mathbf{0}$, i.e. $M_1 \cap M_2= \mathbf{\{0\}}$. Again, $\|T\| < 1$ implies that $D_T$ is invertible. Thus, we have
\[
  \dim(\HS) = \dim(\mathfrak{D}_{T}) = \dim(M_i), \quad i=1,2, \quad \text{and} \quad M_1 \cap M_2 = \{\mathbf{0}\}.
\]
So, $\dim(M_1^{\perp})= \dim(M_2^{\perp})$ and thus there is a unitary between $M_1^{\perp}$ and $M_2^{\perp}$. Considering the unitary $Q$ as in (\ref{eqn:201-A}), we say that there is a unitary $S: \HS \oplus_2 \HS \to \HS \oplus_2 \HS$ extending $Q$. Therefore, $S$ is a unitary that satisfies
\begin{equation}\label{eq:202}
  S(D_{T_1}T_2h, D_{T_2}h)= (D_{T_1}h, D_{T_2}T_1h), \quad \quad h\in \HS.
\end{equation}
Now, consider the Hilbert space $\KS=\HS \oplus_2 \ell_2(\HS\oplus_2\HS)$ and linear operators $V_i: \KS\to \KS$ ($i=1,2$) defined by
\begin{equation}\label{eq:203}
\begin{gathered}
  V_1(h, (h_1,h_2), (h_3, h_4), \dotsc) := (T_1h, S(D_{T_1}h, h_2), S(h_1, h_4), S(h_3, h_6), \dotsc) \\
  V_2(h, (h_1,h_2), (h_3, h_4), \dotsc) := (T_2h, (h_1', D_{T_2}h), (h_3', h_2'), (h_5', h_4'), \dotsc),
  \end{gathered}
\end{equation}
where $h_1', h_2', \dots$ are given by the following non-canonical way: $(h_{2n-1}',h_{2n}')=S^{-1}(h_{2n-1},h_{2n})$ for $n\geq 1$. We shall see that $V_1, V_2$ are isometries as both $S$ and $S^{-1}$ are isometries. Indeed, for any $(h, (h_1,h_2), (h_3, h_4), \dotsc)\in \HS \oplus_2\ell_2(\HS \oplus_2 \HS)$ we have
\begin{align*}
\| V_1(h, (h_1,h_2), (h_3, h_4), \dotsc)\|^2 & = \|T_1h\|^2 + \|S(D_{T_1}h, h_2)\|^2 + \|S(h_1, h_4)\|^2 + \|S(h_3, h_6)\|^2 + \cdots \\
& = \|T_1h\|^2 + \|D_{T_1}h\|^2 + \|h_2\|^2 + \sum_{n=1}^\infty (\|h_{2n-1}\|^2 + \|h_{2n+2}\|^2) \\
& =\|h\|^2 + \|h_2\|^2 + \sum_{n=1}^\infty (\|h_{2n-1}\|^2 + \|h_{2n+2}\|^2) \\
& = \|(h, (h_1,h_2), (h_3, h_4), \dotsc)\|^2,
\end{align*}
where the second equality follows from the fact that $S$ is a unitary. Again,

\begin{align*}
\| V_2(h, (h_1,h_2), (h_3, h_4), \dotsc)\|^2 & = \|(T_2h, (h_1', D_{T_2}h), (h_3', h_2'), (h_5', h_4'), \dotsc)\|^2 \\
& = \|T_2h\|^2 + \|D_{T_2}h\|^2 + \|h_1'\|^2 + \sum_{n=1}^\infty (\|h_{2n+1}'\|^2 + \|h_{2n}'\|^2) \\
& =\|h\|^2 + \sum_{n=1}^\infty (\|h_{2n-1}'\|^2 + \|h_{2n}'\|^2) \\
& = \|h\|^2 + \sum_{n=1}^\infty (\|h_{2n-1}\|^2 + \|h_{2n}\|^2) \tag{\text{as $S^{-1}$ is unitary and $(h_{2n-1}',h_{2n}')=S^{-1}(h_{2n-1},h_{2n}), ~ n\in \N$}}\\
& = \|(h, (h_1,h_2), (h_3, h_4), \dotsc)\|^2
\end{align*}
Now, for every $n\in \N$, the isometry $V_i^n$ is of the form
\begin{equation} \label{eqn:new-002}
  V_i^{n}(h, (h_1,h_2), \dotsc) = (T_i^nh,*, *,\dotsc, ), \quad \text{ for all} \quad (h, (h_1,h_2), \dotsc)\in \KS, \quad i=1,2,
\end{equation}
where the symbol $*$ represents an element of $\HS \oplus_2 \HS$. Consequently, we have
\[
 P_{_{\HS}} V_1^{n_1}V_2^{n_2}h= T_1^{n_1}T_2^{n_2}h, \quad n_1,n_2\in \N \cup\{0\}, \quad h\in \HS,
\]
and consequently $(V_1,V_2)$ is an isometric dilation of $(T_1,T_2)$.
Next, we show that $V_1$ and $V_2$ commute. For each $(h, (h_1,h_2), (h_3,h_4),\dotsc)\in \KS$, we have
\begin{align}\label{eq:204}
  V_1 V_2(h, (h_1,h_2), \dotsc) & = V_1 (T_2h, (h_1', D_{T_2}h), (h_3', h_2'), (h_5', h_4'), \dotsc), ~~(h_{2n-1}',h_{2n}')=S^{-1}(h_{2n-1},h_{2n}) \nonumber\\
  & = (T_1T_2h, S(D_{T_1}T_2h, D_{T_2}h), S(h_1',h_2'), S(h_3',h_4'),\dotsc) \nonumber\\
  & = (T_1T_2h, S(D_{T_1}T_2h, D_{T_2}h), (h_1,h_2), (h_3,h_4),\dotsc)
\end{align}
and 
\begin{align*}
V_2V_1(h, (h_1,h_2), (h_3, h_4), \dotsc) & = V_2 (T_1h, S(D_{T_1}h, h_2), S(h_1, h_4), S(h_3, h_6), \dotsc) \\
  & = V_2 (T_1h, (k_1, k_2), (k_3, k_4), (k_5, k_6), \dotsc),
\end{align*}
where $(k_1,k_2)= S(D_{T_1}h, h_2)$ and $(k_{2n+1},k_{2n+2})= S(h_{2n-1},h_{2n+2})$ for $n\geq 1$. Therefore, it follows from the definition of $V_2$ (as in \eqref{eq:203}) that
\begin{align}
 V_2V_1(h, (h_1,h_2), (h_3, h_4), \dotsc) & = V_2 (T_1h, (k_1, k_2), (k_3, k_4), (k_5, k_6), \dotsc) \nonumber \\
  & = (T_1T_2h, (D_{T_1}h, D_{T_2}T_1h), (h_1, h_2), (h_3, h_4), \dotsc)\label{eq:205},
\end{align}
since $S^{-1}(k_1,k_2)=(D_{T_1}h, h_2)$ and $S^{-1}(k_{2n+1},k_{2n+2})= (h_{2n-1},h_{2n+2})$ for $n\geq 1$. Thus, from \eqref{eq:204}, \eqref{eq:205} and \eqref{eq:202} we have that $V_1V_2 = V_2 V_1$. 

\medskip

Now, we show that the isometric dilation $(V_1, V_2)$ on $\HS \oplus_2\ell_2(\HS \oplus_2 \HS)$ of $(T_1, T_2)$ is minimal, that is, 
\begin{align*}
  \HS \oplus_2\ell_2(\HS \oplus_2 \HS) & =\overline{span}\left\{V_1^{n_1}V_2^{n_2}h: h\in \HS,~ n_1, n_2 \in \N \cup \{0\}\right\}\\ 
  & = \overline{span}\left\{V_1^{n_1}V_2^{n_2}(h, (\mathbf{0}, \mathbf{0}), (\mathbf{0}, \mathbf{0}), \dotsc ): h\in \HS,~ n_1, n_2 \in \N \cup \{0\}\right\}.
\end{align*}
Set 
\[
\mathcal{K}_0 = span\left\{V_1^{n_1}V_2^{n_2}(h, (\mathbf{0}, \mathbf{0}), (\mathbf{0}, \mathbf{0}), \dotsc ): h\in \HS,~ n_1, n_2 \in \N \cup \{0\}\right\}.
\]
We show that $\overline{\mathcal K}_0=\HS \oplus_2\ell_2(\HS \oplus_2 \HS)$. It is evident that the elements of the form $(h, (\mathbf{0}, \mathbf{0}),\dotsc)$ are in $\mathcal{K}_0$ for all $h\in \HS$. Thus, for proving $\HS \oplus_2\ell_2(\HS \oplus_2 \HS)$ is a minimal isometric dilation space, it suffices to show that
\begin{equation}\label{9.eq:new1}
 (\mathbf{0}, \dotsc, (\mathbf{0}, \mathbf{0}), \underbrace{(h_1, h_2)}_{n\text{-th}}, (\mathbf{0}, \mathbf{0}), \dotsc )\in \mathcal{K}_0, \quad \text{ for all } n\in \N, \quad \& \quad \text{for all} \ \ h_1, h_2\in \HS.
\end{equation}
We prove Equation \eqref{9.eq:new1} by mathematical induction on $n$. First, we prove the initial case $n=1$, i.e.,
\begin{equation}\label{9.eq:new2}
(\mathbf{0}, (h_1, h_2), (\mathbf{0}, \mathbf{0}), \dotsc )\in \mathcal{K}_0, \quad \text{ for all } h_1, h_2\in \HS.
\end{equation}
To this end, first we show that
\begin{equation}\label{9.eq:new3}
   (\mathbf{0}, (h,\mathbf{0}), (\mathbf{0}, \mathbf{0}), \dotsc )\in \mathcal{K}_0, \quad \text{ and }\quad (\mathbf{0}, (\mathbf{0},h), (\mathbf{0}, \mathbf{0}), \dotsc )\in \mathcal{K}_0 \quad \text{ for all }h\in \HS.
\end{equation}
Let $h\in \HS$ be arbitrary. Since $D_{T_i}$ is invertible for $i=1,2$, there exist $h_1', h_2'\in \HS$ such that $D_{T_1}h_1' = h=D_{T_2}h_2'$. Then we have 
\begin{align}
(\mathbf{0}, (h,\mathbf{0}), (\mathbf{0}, \mathbf{0}), \dotsc ) & = (\mathbf{0}, (D_{T_1}h_1',\mathbf{0}), (\mathbf{0}, \mathbf{0}), \dotsc ) \nonumber\\
& = (T_1T_2h_1', (D_{T_1}h_1', D_{T_2}T_1h_1'), (\mathbf{0}, \mathbf{0}),\dotsc ) - (T_1T_2h_1', (\mathbf{0}, D_{T_2}T_1h_1'),(\mathbf{0}, \mathbf{0}),\dotsc ) \nonumber\\
& = (T_1T_2h_1', S(D_{T_1}T_2h_1', D_{T_2}h_1'), (\mathbf{0}, \mathbf{0}),\dotsc ) - (T_2T_1h_1', (\mathbf{0}, D_{T_2}T_1h_1', ),(\mathbf{0}, \mathbf{0}),\dotsc ) \tag{$\text{by \eqref{eq:202}}$}\nonumber\\
& = V_1(T_2h_1', (\mathbf{0}, D_{T_2}h_1'), (\mathbf{0}, \mathbf{0}),\dotsc ) - V_2 (T_1h_1', (\mathbf{0}, \mathbf{0}),\dotsc ) \nonumber\\
& = V_1V_2 (h_1', (\mathbf{0}, \mathbf{0}),\dotsc )  - V_2 (T_1h_1', (\mathbf{0}, \mathbf{0}),\dotsc ),\label{9.eq:new4}
\end{align}
and 
\begin{align}\label{9.eq:new5}
 (\mathbf{0}, (\mathbf{0}, h), (\mathbf{0}, \mathbf{0}), \dotsc ) & = (\mathbf{0}, (\mathbf{0}, D_{T_2}h_2'), (\mathbf{0}, \mathbf{0}), \dotsc ) \nonumber\\
  & = (T_2h_2', (\mathbf{0}, D_{T_2}h_2'), (\mathbf{0}, \mathbf{0}), \dotsc ) - (T_2h_2', (\mathbf{0}, \mathbf{0}),\dotsc ) \nonumber\\
  & = V_2 (h_2', (\mathbf{0}, \mathbf{0}),\dotsc ) -(T_2h_2', (\mathbf{0}, \mathbf{0}),\dotsc).
\end{align}
Therefore, \eqref{9.eq:new3} follows from \eqref{9.eq:new4} and \eqref{9.eq:new5}. Now, for all $h_1, h_2\in \HS$, we have
\begin{align*}
(\mathbf{0}, (h_1, h_2), (\mathbf{0}, \mathbf{0}), \dotsc ) = (\mathbf{0}, (h_1, \mathbf{0}), (\mathbf{0}, \mathbf{0}), \dotsc ) + (\mathbf{0}, (\mathbf{0}, h_2), (\mathbf{0}, \mathbf{0}), \dotsc ).
\end{align*}
Consequently, \eqref{9.eq:new2} follows from \eqref{9.eq:new3}, that is, \eqref{9.eq:new1} is true for $n=1$.

\medskip

Let us assume that \eqref{9.eq:new1} is true for $n=2, \dots , j$ and we prove it for $n=j+1$. Let $h_1, h_2\in \HS$ be arbitrary. We need to show that
\begin{equation}\label{9.eq:new6}
 (\mathbf{0}, \dotsc, (\mathbf{0}, \mathbf{0}), \underbrace{(h_1,h_2)}_{(j+1)\text{-th}}, (\mathbf{0}, \mathbf{0}), \dotsc) \in \mathcal{K}_0.
\end{equation}
Let $S^{-1}(h_1, h_2)= (k_1, k_2)\in \HS \oplus_2 \HS$ and consider the element
\[
   a = (\mathbf{0}, \dotsc, (\mathbf{0}, \mathbf{0}), \underbrace{(k_1,\mathbf{0})}_{j\text{-th}}, \underbrace{(\mathbf{0}, k_2)}_{(j+1)\text{-th}}, (\mathbf{0}, \mathbf{0}), \dotsc).
\]
Since $j\geq 2$, it follows from the definition of $V_1$ that
\[
  V_1 (a) = (\mathbf{0}, \dotsc, (\mathbf{0}, \mathbf{0}), \underbrace{S(k_1,k_2)}_{(j+1)\text{-th}}, (\mathbf{0}, \mathbf{0}), \dotsc) = (\mathbf{0}, \dotsc, (\mathbf{0}, \mathbf{0}), \underbrace{(h_1,h_2)}_{(j+1)\text{-th}}, (\mathbf{0}, \mathbf{0}), \dotsc).
\]
Thus, it suffices to show $a\in \mathcal{K}_0$ to establish \eqref{9.eq:new6}. Consider the element
\[
  b= (\mathbf{0}, \dotsc, (\mathbf{0}, \mathbf{0}), \underbrace{S(\mathbf{0}, k_2)}_{j\text{-th}}, (\mathbf{0}, \mathbf{0}), \dotsc).
\]
Then, by assumption we have that $b\in \mathcal{K}_0$, and it follows from the definition of $V_2$ that
\begin{equation}\label{9.eq:new7}
  V_2 (b) = (\mathbf{0}, \dotsc, (\mathbf{0}, \mathbf{0}), \underbrace{(\mathbf{0}, k_2)}_{(j+1)\text{-th}}, (\mathbf{0}, \mathbf{0}), \dotsc).
\end{equation}
Note that
\begin{align*}
  a & = (\mathbf{0}, \dotsc, (\mathbf{0}, \mathbf{0}), \underbrace{(k_1,\mathbf{0})}_{j\text{-th}}, (\mathbf{0}, \mathbf{0}), \dotsc) +  (\mathbf{0}, \dotsc, (\mathbf{0}, \mathbf{0}), \underbrace{(\mathbf{0}, k_2)}_{(j+1)\text{-th}}, (\mathbf{0}, \mathbf{0}), \dotsc) \\
  & = (\mathbf{0}, \dotsc, (\mathbf{0}, \mathbf{0}), \underbrace{(k_1,\mathbf{0})}_{j\text{-th}}, (\mathbf{0}, \mathbf{0}), \dotsc) + V_2(b), \quad \text{and} \quad b\in \mathcal{K}_0\tag{$\text{by \eqref{9.eq:new7}}$}.
\end{align*}
Therefore, $a\in \mathcal{K}_0$. So, by mathematical induction, $(V_1, V_2)$ on $\HS \oplus_2 \ell_2(\HS \oplus_2 \HS)$ is a minimal isometric dilation of $(T_1, T_2)$. This completes the proof.
\end{proof}

As is clear from the proof of Theorem \ref{thm:201} that we obtained a minimal isometric dilation on the space $\HS \oplus_2 \ell_2(\HS \oplus_2 \HS)$ by eliminating the alternate zeros from And\^{o} unitary $U$ on $\HS \oplus_2 \HS \oplus_2 \HS \oplus_2 \HS$ as in Equation-(\ref{eq:102}). Consequently, we found a suitable unitary $S$ on $\HS \oplus_2 \HS$ (see Equation-(\ref{eq:202})) that plays the role of $U$ in our case and satisfies
\begin{equation}\label{eq:104}
  S(D_{T_1}T_2h , D_{T_2}h) = (D_{T_1}h, D_{T_2}T_1h), \quad h\in \HS.
\end{equation}
We learn from \cite[Lemma 7.2]{JPR} that whenever $\|T\|<1$ and $A_T$ defines a norm on a Banach space $\X$, the space $(\X, A_T)$ is also a Banach space. This is an elementary routine exercise and we shall use this result. Since \cite{JPR} is not yet a published or accepted paper, we include its proof here from \cite{JPR} in a copy-paste manner.
\begin{lem}\cite[Lemma 7.2]{JPR}\label{lem: Equv. of norm}
Let $\X$ be a Banach space and $T \in \mathcal{B}(\X)$ be a strict contraction such that the function $A_T$ as in $(\ref{eqn:new-Sec1-001})$ defines a norm on $\X$. Then $\X_0 = (\X, A_T)$ is a Banach space and $A_T$ is equivalent to the original norm on $\X$.
\end{lem}

\begin{proof}
Let $\|.\|$ be the original norm on $\X$. Suppose $(x_n)$ is a Cauchy sequence in $\X_0$. Then for every $\epsilon >0$, there is $n_0 \in \mathbb{N}$ such that
\begin{align*}
    & \left( \|x_n-x_m\|^2 - \|T(x_n - x_m) \|^2 \right)^{ 1\slash 2} < \epsilon \left(1 - \|T\|^2\right)^{1 \slash 2},\quad n, ~m \geq n_0, \\
    &  \Rightarrow \|x_n-x_m\|^2 < \epsilon^2 (1 - \|T\|^2) + \|T\|^2 \|x_n-x_m\|^2, \quad n,~ m \geq n_0,  \\
    & \Rightarrow \|x_n-x_m\|^2 (1-\|T\|^2) < \epsilon^2 (1 - \|T\|^2), \quad n, ~ m \geq n_0,  \\
    &  \Rightarrow \|x_n-x_m \| < \epsilon, \quad n, ~ m \geq n_0, \qquad \left[ ~\text{ as }~ (1- \|T\|^2) >0 \right]
\end{align*}
which shows that $(x_n)$ is a Cauchy sequence in $(\X, \|.\|)$. Since $(\X, \|.\|)$ is Banach space, there is $x\in \X$ such that $x_n \to x$ with respect to $\|.\|$ as $n\to \infty$. We now show that the Cauchy sequence $(x_n)$ converges to $x$ in $\X_0$. By the continuity of $T$ and the norm, we have that
\[
  A_T(x_n -x) = \left( \|x_n-x\|^2 - \|T(x_n -x) \|^2 \right)^{ 1\slash 2} \longrightarrow 0, \quad \text{as} \quad n \to \infty.
\]
Thus, $\X_0$ is a Banach space. Moreover, we have $A_T(x) \leq \|x\|$ for all $x\in \X$, which shows that the identity map $I_d: \X \to \X_0$ is a bounded bijective linear map. Therefore, by open mapping theorem, $I_d^{-1}: \X_0 \to \X$ is also a bounded linear map. Therefore, there exists a constant $k >0$ namely $k=\|I_d^{-1}\|$ such that $A_T(x) \leq \|x\| \leq k~  A_T(x)$, $\ x\in \X$. Thus, the norm $A_T$ on $\X$ is equivalent to the original norm $\|.\|$ on $\X$ and the proof is complete.
\end{proof}

Since $(\X, A_T)$ is a Banach space, in order to switch from Hilbert space to Banach space setting, first we replace the defect spaces $\mathfrak{D}_{T_i}$ (which is equal to $\HS$ when $T_i$ is a strict contraction) by the respective Banach spaces $\X_i=\left(\X,A_{T_i}\right)$ for $i=1,2$. The task is not complete unless we find a unitary operator $S: \X_1 \oplus_2 \X_2 \to \X_1 \oplus_2 \X_2$ satisfying
\begin{equation}\label{eq:105}
   S(T_2x, x) = (x,T_1x), \quad x\in \X,
\end{equation}
which is analogous to \eqref{eq:104}. However, unlike in Hilbert space setting such a unitary $S:\X_1\oplus_2 \X_2 \to \X_1\oplus_2 \X_2$ satisfying \eqref{eq:105} may not exist as a closed linear subspace of Banach space is not always complemented. For this reason, we are bound to assume the existence of such a unitary $S$ as a condition to reach at our desired dilation. However, we shall see that assuming this as a condition leads to an analogous minimal isometric dilation on $\X \oplus_2 \ell_2(\X_1 \oplus_2 \X_2)$ for a commuting pair of strict contractions on a Banach space $\X$. We begin with a preparatory lemma that records in Banach space setting the properties of the subspaces analogous to $M_1$ and $M_2$ considered in Equation-(\ref{eqn:new-021}) in the proof of Theorem \ref{thm:201}.

\begin{lem}\label{lem:201}
Let $(T_1,T_2)$ be a commuting pair of contractions on a normed linear space $\X$ such that the function $A_{T_i}: \X \to [0,\infty)$ defines a norm on $\X$ for $i=1,2$. Denote $\X_i= \left(\X, A_{T_i}\right)$ for $i=1,2$. Consider the subspaces $\widehat{M}_1, \widehat{M}_2\subseteq \X_1 \oplus_2 \X_2$ given by
\begin{equation}\label{eq:new.v3.1}
  \widehat{M}_1= \{(T_2x,x): x\in \X\}, \quad \& \quad \widehat{M}_2=\{(x,T_1x): x\in \X\}.
\end{equation}

If $\|T_1T_2\| < 1$, then the following hold:
\begin{enumerate}
\item[(i)] the function $A_T:\X \to [0,\infty)$ induces a norm on $\X$, where $T= T_1T_2$;

\item[(ii)] $\widehat{M}_1 \cap \widehat{M}_2 = \{\mathbf{0}\}$ and in addition if $\X$ is a Banach space, then both $\widehat{M}_1$ and $\widehat{M}_2$ are closed in $\X_1 \oplus_2 \X_2$.
\end{enumerate}
Moreover, if there  exist unitarily equivalent subspaces $\Y_1$, $\Y_2$ of $\X_1 \oplus_2 \X_2$ such that 
\begin{equation}\label{eq:new.v3.2}
 \widehat{M}_1\oplus_2 \Y_1 = \X_1 \oplus_2 \X_2= \widehat{M}_2\oplus_2 \Y_2,
\end{equation}
then there is a unitary $S: \X_1 \oplus_2 \X_2 \to \X_1 \oplus_2 \X_2$ such that $S(\widehat{M}_1) = \widehat{M}_2$, i.e.,
\begin{equation}\label{eq:new.v3.3}
    S(T_2x, x) = (x, T_1x), \quad x\in \X.
\end{equation}

\end{lem}

\begin{proof}
$(i)$ The homogeneity, non-negativity and positivity of $A_T$ is trivial, where $T=T_1T_2$. We only prove the triangle inequality. For all $x\in \X$, we have
\[
  A_T(x) = \left(\|x\|^2 - \|T_1T_2x\|^2\right)^{\frac{1}{2}} = \left(\|x\|^2 - \|T_1x\|^2 + \|T_1x\|^2 - \|T_1T_2x\|^2\right)^{\frac{1}{2}} = \left(A_{T_1}(x)^2 + A_{T_2}(T_1x)^2\right)^{\frac{1}{2}}.
\]
If $\|\cdot\|_1$ and $\|\cdot\|_2$ are two norms on a normed linear space, then so is $\left(\|\cdot\|_1^2 + \|\cdot\|_2^2\right)^{\frac{1}{2}}$ on the same space. Hence, $A_T$ is a norm on $\X$.

\smallskip

$(ii)$ To show $\widehat{M}_1 \cap \widehat{M}_2 = \{\mathbf{0}\}$, let $z\in \widehat{M}_1 \cap \widehat{M}_2$. Then there exist $x, y\in \X$ such that 
$
 (x, T_1x) = z = (T_2y, y),
$
which shows that $x= T_1T_2x = Tx$. Since $\|T\|< 1$, we have $x= \mathbf{0}$. Consequently, $z= \mathbf{0}$. Now, assume that $\X$ is a Banach space. Since $\|T\| < 1$, and $A_T$ defines a norm on $\X$, it follows from Lemma \ref{lem: Equv. of norm} that $A_T$ is a complete norm on $\X$. Consider the Banach space $\X_3=\left(\X, A_T\right)$. Then the maps $W_i: \X_3 \to \X_1 \oplus_2 \X_2$ for $i=1,2$ defined by
\[
  W_i(x) = \begin{cases}
   			 (T_2x, x) & \text{ if } i=1 \\
   			 (x, T_1x) & \text{ if } i=2
   			\end{cases}, \quad x\in \X_3,
\]
are linear isometries. Consequently, $W_i(\X_3)=\widehat{M}_i$ is a closed subspace of $\X_1 \oplus_2 \X_2$ for $i=1,2$. 

\smallskip

Now we prove the last part of the theorem. Note that the map $Q: \widehat{M}_1 \to \widehat{M}_2$ defined by $Q(T_2x,x):= (x,T_1x)$, $x\in \X$, is a unitary as
\begin{align*}
  \|Q(T_2x, x)\|^2 = \|x\|^2 - \|T_1x\|^2 + \|T_1x\|^2 - \|T_2T_1 x\|^2
  & = \|x\|^2 - \|T_2x\|^2 + \|T_2x\|^2 - \|T_1T_2 x\|^2 \\
  & = \|(T_2x, x)\|^2.
\end{align*} 
Let $S_1:\Y_1 \to \Y_2$ be the unitary, where the subspaces $\Y_1$ and $\Y_2$ are as in the hypothesis. Then the operator $S: \X_1\oplus_2 \X_2 \to \X_1\oplus_2 \X_2$ defined by
$
  S(m_1+y_1) = Q(m_1) + S_1(y_1)$, $ m_1\in \widehat{M}_1, \ y_1\in \Y_1$,
is a unitary and satisfies the property as stated. This completes the proof.
\end{proof}

A converse to the last part of Lemma \ref{lem:201} does not hold in general, which is to say that the existence of a unitary $S$ satisfying \eqref{eq:new.v3.3} does not guarantee the existence of unitarily equivalent subspaces $\Y_1$ and $\Y_2$ satisfying \eqref{eq:new.v3.2}. Once again, the reason is that the subspaces $\widehat{M}_1, \widehat{M}_2$ may not be complemented in $\X_1 \oplus_2 \X_2$. Now, we are in a position of presenting our conditional isometric dilation of a pair of strict Banach space contractions.

\begin{thm}\label{thm:202}
Let $(T_1,T_2)$ be a commuting pair of strict contractions on a Banach space $\X$ such that the function $A_{T_i}: \X \to [0,\infty)$ given by $(\ref{eqn:new-Sec1-001})$, i.e., $A_{T_i}(x) = \left(\|x\|^2 - \|T_ix\|^2\right)^{\frac{1}{2}}$, $ x\in \X$, defines a norm on $\X$ for $i=1,2$. Let $\X_i=\left(\X, A_{T_i}\right)$ for $i=1,2$. If there is a unitary $S: \X_1 \oplus_2 \X_2 \to \X_1 \oplus_2 \X_2$ satisfying
\begin{equation}\label{eq:206}
  S(T_2x, x) = (x, T_1x), \quad x\in \X \ ,
\end{equation}
then $(T_1,T_2)$ admits a minimal isometric dilation on $\X \oplus_2 \ell_2(\X_1 \oplus_2 \X_2)$.
\end{thm}

\begin{proof}
The proof we present here is analogous to that of Theorem \ref{thm:201}, only we avoid the adjoint of an operator. Since $\|T_i\| < 1$ and $A_{T_i}$ defines a norm on $\X$, it follows from Lemma \ref{lem: Equv. of norm} that the space $\X_i= (\X, A_{T_i})$ is a Banach space for $i=1,2$. Consequently, the space $\X_1 \oplus_2 \X_2$ is also a Banach space. Consider the Banach space $\widetilde{\X}= \X \oplus_2 \ell_2(\X_1 \oplus_2 \X_2)$, and the operators $V_i:\widetilde{\X} \to \widetilde{\X}$ ($i=1,2$) defined by
\begin{equation}\label{eq:207}
\begin{gathered}
V_1(x, (x_1,x_2), (x_3,x_4), \dotsc) = (T_1x, S(x, x_2), S(x_1,x_4), S(x_3,x_6), \dotsc) \\
V_2(x, (x_1,x_2), (x_3,x_4), \dotsc) = (T_2x, (x_1', x), (x_3',x_2'), (x_5',x_4'), \dotsc), 
\end{gathered}
\end{equation}
where $x_1', x_2', \dots$ are given by $(x_{2n-1}', x_{2n}')= S^{-1}(x_{2n-1},x_{2n})$ for $n\geq 1$. Then, by an argument similar to that in the proof of Theorem \ref{thm:201}, $V_i$ is a linear isometry for $i=1,2$ as both $S$ and $S^{-1}$ are linear isometries. Also, following Equation-(\ref{eqn:new-002}) in the proof of Theorem \ref{thm:201}, we see that for all $n\in \N$ the operator $V_i^n$ for $i=1,2$ has the following form
\[
  V_i^{n}(x, (x_1,x_2), \dotsc) = (T_i^nx,*, *,\dotsc, ), \quad \text{ for all } \quad (x, (x_1,x_2), \dotsc)\in \widetilde{\X}, \quad i=1,2,
\]
where the symbol $*$ represents an element of $\X_1 \oplus_2 \X_2$. Consequently, we have
\[
  P_{_{\X}}V_1^{n_1}V_2^{n_2}x= T_1^{n_1}T_2^{n_2}x, \quad x\in \X, \quad n_i\in \N \cup \{0\}, ~~i=1,2.
\] 
Now, we show that $V_1$ and $V_2$ commutes. For each $(x, (x_1,x_2), (x_3,x_4),\dotsc)\in \widetilde{\X}$, we have
\begin{align}\label{eq:208}
  V_1 V_2(x, (x_1,x_2), (x_3, x_4), \dotsc) & = V_1 (T_2x, (x_1', x), (x_3', x_2'), (x_5', x_4'), \dotsc), ~~ (x_{2n-1}', x_{2n}') = S^{-1}(x_{2n-1},x_{2n})\nonumber\\
  & = (T_1T_2x, S(T_2x, x), S(x_1',x_2'), S(x_3',x_4'),\dotsc) \nonumber\\
  & = (T_1T_2x, S(T_2x, x), (x_1,x_2), (x_3,x_4),\dotsc)
\end{align}
and 
\begin{align*}
V_2V_1(x, (x_1,x_2), (x_3, x_4), \dotsc) & = V_2 (T_1x, S(x, x_2), S(x_1, x_4), S(x_3, x_6), \dotsc) \\
  & = V_2 (T_1x, (y_1, y_2), (y_3, y_4), (y_5, y_6), \dotsc)~~\text{ (say),}
\end{align*}
where $(y_1,y_2)= S(x, x_2)$ and $(y_{2n+1},y_{2n+2})= S(x_{2n-1},x_{2n+2})$ for $n\geq 1$. Therefore, it follows from the definition of $V_2$ (as in \eqref{eq:207}) that
\begin{align}
 V_2V_1(x, (x_1,x_2), (x_3, x_4), \dotsc) & = V_2 (T_1x, (y_1, y_2), (y_3, y_4), (y_5, y_6), \dotsc) \nonumber \\
  & = (T_1T_2x, (x, T_1x), (x_1, x_2), (x_3, x_4), \dotsc)\label{eq:209},
\end{align}
since $S^{-1}(y_1,y_2)=(x, x_2)$ and $S^{-1}(y_{2n+1},y_{2n+2})= (x_{2n-1},x_{2n+2})$ for $n\geq 1$. Consequently, it follows from \eqref{eq:208}, \eqref{eq:209} and \eqref{eq:206} that $V_1V_2 = V_2 V_1$.

\medskip

We now show that the dilation $(V_1, V_2)$ on $\widetilde{\X}= \X \oplus_2\ell_2(\X_1 \oplus_2 \X_2)$ is minimal, that is,
\begin{align*}
  \X \oplus_2\ell_2(\X_1 \oplus_2 \X_2) & = \overline{span}\left\{V_1^{n_1}V_2^{n_2}x: x\in \X,~ n_1, n_2 \in \N \cup \{0\}\right\} \\
  & = \overline{span}\left\{V_1^{n_1}V_2^{n_2}(x, (\mathbf{0}, \mathbf{0}), (\mathbf{0}, \mathbf{0}), \dotsc ): x\in \X,~ n_1, n_2 \in \N \cup \{0\}\right\}.
\end{align*}
Once again, the proof of this part is also analogous to the same part of Theorem \ref{thm:201}. Let us set
\[
\widetilde{\X}_0 = span\left\{V_1^{n_1}V_2^{n_2}(x, (\mathbf{0}, \mathbf{0}), (\mathbf{0}, \mathbf{0}), \dotsc ): x\in \X,~ n_1, n_2 \in \N \cup \{0\}\right\}.
\]
We need to show that $\overline{\widetilde{\X}}_0=\X \oplus_2\ell_2(\X_1 \oplus_2 \X_2)$. It is evident that the elements of the form $(x, (\mathbf{0}, \mathbf{0}),\dotsc)$ are in $\widetilde{\X}_0$ for all $x\in \X$. Thus, it suffices to show that
\begin{equation}\label{9.eq:new8}
 (\mathbf{0}, \dotsc, (\mathbf{0}, \mathbf{0}), \underbrace{(x_1, x_2)}_{n\text{-th}}, (\mathbf{0}, \mathbf{0}), \dotsc )\in \widetilde{\X}_0, \quad \text{ for all } n\in \N, \quad \text{ and } \quad x_i\in \X_i, ~i=1,2.
\end{equation}
We prove this by mathematical induction on $n$. First, we prove the initial case $n=1$, i.e.,
\begin{equation}\label{9.eq:new9}
(\mathbf{0}, (x_1, x_2), (\mathbf{0}, \mathbf{0}), \dotsc )\in \widetilde{\X}_0, \quad \forall x_i\in \X_i, ~i=1,2.
\end{equation}
To do so, first we show that
\begin{equation}\label{9.eq:new10}
   (\mathbf{0}, (x_1,\mathbf{0}), (\mathbf{0}, \mathbf{0}), \dotsc )\in \widetilde{\X}_0, \quad \text{ and }\quad (\mathbf{0}, (\mathbf{0},x_2), (\mathbf{0}, \mathbf{0}), \dotsc )\in \widetilde{\X}_0 \quad \forall x_i\in \X_i, ~i=1,2.
\end{equation}
Let $x_i\in \X_i$ ($i=1,2$) be arbitrary. Then we have 
\begin{align}
(\mathbf{0}, (x_1,\mathbf{0}), (\mathbf{0}, \mathbf{0}), \dotsc ) & = (T_1T_2x_1, (x_1, T_1x_1), (\mathbf{0}, \mathbf{0}),\dotsc ) - (T_1T_2x_1, (\mathbf{0}, T_1x_1),(\mathbf{0}, \mathbf{0}),\dotsc ) \nonumber\\
& = (T_1T_2x_1, S(T_2x_1, x_1), (\mathbf{0}, \mathbf{0}),\dotsc ) - (T_2T_1x_1, (\mathbf{0},T_1x_1, ),(\mathbf{0}, \mathbf{0}),\dotsc ) \tag{$\text{by \eqref{eq:206}}$}\nonumber\\
& = V_1(T_2x_1, (\mathbf{0}, x_1), (\mathbf{0}, \mathbf{0}),\dotsc ) - V_2 (T_1x_1, (\mathbf{0}, \mathbf{0}),\dotsc ) \nonumber\\
& = V_1V_2 (x_1, (\mathbf{0}, \mathbf{0}),\dotsc )  - V_2 (T_1x_1, (\mathbf{0}, \mathbf{0}),\dotsc ),\label{9.eq:new11}
\end{align}
and 
\begin{align}\label{9.eq:new12}
 (\mathbf{0}, (\mathbf{0}, x_2), (\mathbf{0}, \mathbf{0}), \dotsc ) & = (T_2x_2, (\mathbf{0}, x_2), (\mathbf{0}, \mathbf{0}), \dotsc ) - (T_2x_2, (\mathbf{0}, \mathbf{0}),\dotsc ) \nonumber\\
  & = V_2 (x_2, (\mathbf{0}, \mathbf{0}),\dotsc ) -(T_2x_2, (\mathbf{0}, \mathbf{0}),\dotsc).
\end{align}
Therefore, \eqref{9.eq:new10} follows from \eqref{9.eq:new11} and \eqref{9.eq:new12}. Now, for any $x_1\in \X_1$ and $x_2 \in \X_2$, we have
\begin{align*}
(\mathbf{0}, (x_1, x_2), (\mathbf{0}, \mathbf{0}), \dotsc ) = (\mathbf{0}, (x_1, \mathbf{0}), (\mathbf{0}, \mathbf{0}), \dotsc ) + (\mathbf{0}, (\mathbf{0}, x_2), (\mathbf{0}, \mathbf{0}), \dotsc ).
\end{align*}
Consequently, \eqref{9.eq:new9} follows from \eqref{9.eq:new10}, that is, \eqref{9.eq:new8} is true for $n=1$.

\medskip

Now, let us assume that \eqref{9.eq:new8} is true for $n=2, \dots , j$ and then let us prove that it is true for $n=j+1$. Let $x_i\in \X_i$ be arbitrary for $i=1,2$. We show that
\begin{equation}\label{9.eq:new13}
 (\mathbf{0}, \dotsc, (\mathbf{0}, \mathbf{0}), \underbrace{(x_1,x_2)}_{(j+1)\text{-th}}, (\mathbf{0}, \mathbf{0}), \dotsc) \in \widetilde{\X}_0.
\end{equation}
Let $(x_1',x_2')= S^{-1}(x_1, x_2)\in \X_1 \oplus_2 \X_2$ and consider the element
\[
   a = (\mathbf{0}, \dotsc, (\mathbf{0}, \mathbf{0}), \underbrace{(x_1',\mathbf{0})}_{j\text{-th}}, \underbrace{(\mathbf{0}, x_2')}_{(j+1)\text{-th}}, (\mathbf{0}, \mathbf{0}), \dotsc).
\]
Since $j+1\geq 2$, it follows from the definition of $V_1$ that
\[
  V_1 (a) = (\mathbf{0}, \dotsc, (\mathbf{0}, \mathbf{0}), \underbrace{S(x_1',x_2')}_{(j+1)\text{-th}}, (\mathbf{0}, \mathbf{0}), \dotsc) = (\mathbf{0}, \dotsc, (\mathbf{0}, \mathbf{0}), \underbrace{(x_1,x_2)}_{(j+1)\text{-th}}, (\mathbf{0}, \mathbf{0}), \dotsc).
\]
Thus, it suffices to show $a\in \widetilde{\X}_0$ to establish \eqref{9.eq:new13}. Consider the element
\[
  b= (\mathbf{0}, \dotsc, (\mathbf{0}, \mathbf{0}), \underbrace{S(\mathbf{0}, x_2')}_{j\text{-th}}, (\mathbf{0}, \mathbf{0}), \dotsc).
\]
Then by assumption we have $b\in \widetilde{\X}_0$, and it follows from the definition of $V_2$ that
\begin{equation}\label{9.eq:new14}
  V_2 (b) = (\mathbf{0}, \dotsc, (\mathbf{0}, \mathbf{0}), \underbrace{(\mathbf{0}, x_2')}_{(j+1)\text{-th}}, (\mathbf{0}, \mathbf{0}), \dotsc).
\end{equation}
Note that
\begin{align*}
  a & = (\mathbf{0}, \dotsc, (\mathbf{0}, \mathbf{0}), \underbrace{(x_1',\mathbf{0})}_{j\text{-th}}, (\mathbf{0}, \mathbf{0}), \dotsc) +  (\mathbf{0}, \dotsc, (\mathbf{0}, \mathbf{0}), \underbrace{(\mathbf{0}, x_2')}_{(j+1)\text{-th}}, (\mathbf{0}, \mathbf{0}), \dotsc) \\
  & = (\mathbf{0}, \dotsc, (\mathbf{0}, \mathbf{0}), \underbrace{(x_1',\mathbf{0})}_{j\text{-th}}, (\mathbf{0}, \mathbf{0}), \dotsc) + V_2(b), \quad \text{and} \quad b\in \widetilde{\X}_0\tag{$\text{by \eqref{9.eq:new14}}$}.
\end{align*}
Thus, $a\in \widetilde{\X}_0$ and by mathematical induction $(V_1,V_2)$ on $\X \oplus_2 \ell_2(\X_1\oplus_2 \X_2)$ is a minimal isometric dilation of $(T_1,T_2)$. This completes the proof.
\end{proof}

\begin{rem}\label{rem:201}
As discussed before, an isometric dilation as in Theorem \ref{thm:202} is possible if there is a unitary $S: \X_1 \oplus_2 \X_2 \to \X_1 \oplus_2 \X_2$ satisfying \eqref{eq:206}. The last part of Lemma \ref{lem:201} provides a sufficient condition for the existence of such a unitary. A closer look at the unitary $S$ reveals that we are actually seeking a unitary extension on $\X_1 \oplus_2 \X_2$ of the unitary $Q:\widehat{M}_1 \to \widehat{M}_2$ defined by
$ Q(T_2x, x) = (x, T_1x)$, $x\in \X$, where $\widehat{M}_1, \widehat{M}_2$ are given in \eqref{eq:new.v3.1}. It leads to the following general question, which seems to be too difficult to answer at this point.

\medskip

\noindent
\textit{Question.} Let $L_1$, $L_2$ be non-zero closed linear subspaces of a Banach space $\X$ satisfying $L_1 \cap L_2 = \{\mathbf{0}\}$, and let $U: L_1 \to L_2$ be a unitary operator. Is there a unitary $\widetilde{U}: \X \to \X$ that extends $U$?
\end{rem}
We now present two families of examples of strict Banach space contractions that admit dilation as per Theorem \ref{thm:202}.
\begin{eg}
Let $\X$ be a Banach space and let $(V_1, V_2)$ be a commuting pair of unitaries on $\X$. For $c\in \D$, consider $T_2= c V_1$ and $T_1 = V_2 T_2$. Then $(T_1, T_2)$ is a commuting pair of strict contractions on $\X$, and clearly $A_{T_i}$ (see (\ref{eqn:new-Sec1-001}) for definition) induces a norm on $\X$, where $A_{T_i}(x) = (1- |c|^2)^{\frac{1}{2}}\|x\|$  for $i=1,2$. Let $\X_i=(\X, A_{T_i})$, $i=1,2$. So, the operator $U: \X_1 \oplus_2 \X_2 \to \X_1 \oplus_2 \X_2$ defined by $U(x, y) = (y, V_2x)$ is surjective and satisfies $ U(T_2x, x) = (x, V_2 T_2x) = (x, T_1x)$ for any $x\in \X$. Again,
\[
\|T_1x\|=\|V_2T_2x\|=\|V_2(cV_1)x\|=|c|\|x\| \ \ \& \ \ \|T_2x\|=\|cV_1x\|=|c|\|x\|
\]
and thus $U$ is a unitary. Thus, $(T_1, T_2)$ satisfies the hypothesis of Theorem \ref{thm:202}. Now, let us see this with a particular family of Banach spaces $\X$ and pairs $(T_1, T_2)$ acting on it. 

Choose $\X$ to be the Banach space $L^p(\R, \mu)$ consisting of complex-valued $p$-th integrable functions on $\R$ for $1\leq p < \infty$ and $p\neq 2$, where $\mu$ is the Lebesgue measure on $\R$. Let $a, b\in \R$. Define $V_1,V_2$ on $\X$ by $V_1(f)(x) =f(x-a)$, $V_2(f)(x) = f(x-b)$, for $f\in \X$ and $x\in \R$. Then, $V_1,V_2$ are unitaries, as $\mu$ is tranalation invariant. Also, we have $V_1 V_2 (f)(x) = f(x-a-b) = V_2 V_1(f)(x)$ for all $f\in \X$ and thus, $V_1,V_2$ commute. Obviously, one can set $T_2 = cV_1$ for some $c\in \D$ and $T_1 = V_2 T_2$.
\end{eg}

\begin{eg}\label{eg:23:07:26::12:43}
Let $\X$ be a non-Hilbert Banach space and $\HS$ be a Hilbert space so that $\Y = \X \oplus_2 \HS$ is a non-Hilbert Banach space. Let $(T_1, T_2)$ be a commuting pair of strict contractions on $\HS$. Also, let $T$ be a strict contraction on $\X$ such that the function $A_T$ (as in (\ref{eqn:new-Sec1-001})) defines a norm on $\X$. For any $c\in \D$, $cV$ for any isometry $V$ on $\X$ can be such a strict contraction $T$, because, then $A_T(x)=(1-|c|^2)^{\frac{1}{2}}\|x\|$ which clearly defines a norm. Note that $A_{T_i}(h) = \|D_{T_i}h\|$ for all $h\in \HS$, and $i=1,2$. It can be easily verified that each $A_{T_i}$ defines a norm on $\HS$, which is induced by the inner product $(h_1, h_2)\mapsto \ \langle D_{T_i}(h_1),\ D_{T_i}(h_2) \rangle$. Thus, by Lemma \ref{lem: Equv. of norm} $\HS_i=(\HS,\ A_{T_i})$ is a Hilbert space for $i=1,2$. Consider the operators $\widehat{T}_1, \widehat{T}_2$ defined on $\Y$ by
\[
\widehat{T}_1:= \begin{bmatrix}
                  T & \mathbf{0} \\
                  \mathbf{0} & T_1
                  \end{bmatrix} \quad \text{and} \quad \widehat{T}_2 := \begin{bmatrix}
                  T & \mathbf{0} \\
                  \mathbf{0} & T_2
                  \end{bmatrix}.
\]
Then $(\widehat{T}_1, \widehat{T}_2)$ is a commuting pair of strict contractions on $\Y$. Also, for $(x, h)\in \Y$, we have
\begin{align*}
A_{\widehat{T}_i}(x,h) = \left(\|(x, h)\|^2 - \|(Tx, T_ih)\|^2\right)^{\frac{1}{2}} & = \left(\|x\|^2 + \|h\|^2 - \|Tx\|^2 - \|T_ih\|^2\right)^{\frac{1}{2}} \nonumber \\
& = \left(A_{T}(x)^2 + A_{T_i}(h)^2\right)^{\frac{1}{2}}, \quad i=1,2.
\end{align*}
Therefore, both $A_{\widehat{T}_1}$ and $A_{\widehat{T}_2}$ are norms on $\Y$, which follow from the general fact that if $\|\cdot\|_1$ and $\|\cdot\|_2$ are two norms on a Banach space, then so is $\left(\|\cdot\|_1^2 + \|\cdot\|_2^2\right)^{\frac{1}{2}}$. Let $\Y_i = (\Y, \ A_{\widehat{T}_i})$, $i=1,2$.  Now, by an argument similar to that in the proof of Theorem \ref{thm:201}, we find a unitary $S: \HS_1 \oplus_2 \HS_2 \to \HS_1 \oplus_2 \HS_2$ satisfying $S(T_2h, h) = (h, T_1h)$ for all $h\in \HS$. Define $U: \Y_1 \oplus_2 \Y_2 \to \Y_1 \oplus_2 \Y_2$ by 
$
  U((x_1, h_1), (x_2,h_2)) := ((x_2, k_1),(x_1, k_2))$, where $S(h_1, h_2) = (k_1, k_2)$. Then, $U$ is a unitary as $S$ is a unitary and $U(\widehat{T}_2(x,h),(x,h)) = U((Tx, T_2h),(x, h)) = ((x, h), (Tx, T_1h))=((x,h),\widehat{T}_1(x,h))$, where the second equality follows from $S(T_2h, h)= (h, T_1h)$.
Hence, the pair $(\widehat{T}_1, \widehat{T}_2)$ satisfies the hypothesis of Theorem \ref{thm:202}.

\smallskip 

Note that, $T_1,T_2$ can be taken to be any commuting strict contractions on any Hilbert space $\HS$. However, we give an example of a specific non-Hilbert Banach space $\X$ and a strict contraction $T$ on $\X$ such that $A_T$ defines a norm on $\X$. Let $\X= \ell_2(\widehat{\Y})$ for a non-Hilbert Banach space $\widehat{\Y}$. Then $\X$ is a Banach space as was mentioned in the beginning of the `Introduction'. Consider $T = c \widehat{M}_z$, where $c\in \D$ and $\widehat{M}_z$ is the backward shift operator on $\X=\ell_2(\widehat{\Y})$. Then,
\[
  A_T(x_1, x_2, \dotsc) = \left(\|x_1\|^2  + (1-|c|^2)\sum_{n=2}^\infty \|x_n\|^2\right)^{\frac{1}{2}} = \|D(x_1, x_2, \dots)\|,
\]
where $D: \X \to \X$ is evidently a bounded operator, which is defined by
\[
  D(x_1, x_2, \dots) = ( x_1, (1- |c|^2)^{1/2} \ x_2, (1- |c|^2)^{1/2} \ x_3, \cdots ), \quad (x_1, x_2, \dotsc ) \in \X.
\]
Now, it can be easily verified that $A_T$ defines a norm on $\X$.
\end{eg}

\section{Dilation in a more general setting}\label{sec:03}

\noindent
In this section, we show that the construction of isometric dilation for a commuting pair of strict Banach space contractions as in Theorem \ref{thm:202} is valid in a more general frame of normed spaces. The primary reason is that the function $A_{T_i}$ can define a norm on $\X$ even when $\|T_i\|=1$ for at least one $i$. One can find examples of such Banach space contractions in \cite[Subsection 7.1]{JPR}. Also, we see from \cite[Lemma 7.3]{JPR} that if $\X$ is a Banach space and $A_T$ defines a norm on $\X$ for a contraction $T$, then  the normed linear space $\X_0 = \left(\X, A_T\right)$ cannot be a Banach space if $\|T\|=1$. Let $(T_1, T_2)$ be a commuting pair of contractions (not necessarily strict contractions) on a normed linear space $\X$, which is not necessarily a Banach space, such that $A_{T_i}$ defines a norm on $\X$ for $i=1, 2$. Let $\X_i$ be the space $\left(\X, A_{T_i}\right)$ for $i=1,2$ and let us consider the following possible cases.

\begin{enumerate}

\item[(a)] If $\|T_i\|=1$ for both $i=1,2$, then $\|T_1T_2\|$ may or may not be less than $1$. So, in view of \cite[Lemma 7.3]{JPR}, we cannot assert that the subspaces $\widehat{M}_1$ and $\widehat{M}_2$ as in Lemma \ref{lem:201} are closed subspaces. However, in this case we have isometric dilation (see Theorem \ref{thm:301} below) under the same assumptions as in Theorem \ref{thm:202}. Thus, we have an isometric dilation of $(T_1,T_2)$ on a normed linear space which is not a Banach space.

\smallskip

\item[(b)] If one of $\{T_1,T_2\}$ has norm equal to $1$ and the other is a strict contraction, then $\|T_1T_2\|<1$ and part-(i) of Lemma \ref{lem:201} is valid. However, the space $\X \oplus_2 \ell_2(\X_1 \oplus_2 \X_2)$ is no longer a Banach space. Indeed, the fact that $\X \oplus_2 \ell_2(\X_1 \oplus_2 \X_2)$ is a Banach space implies that both $\X_1$ and $\X_2$ are Banach spaces, which is not true at least for the contraction having norm equal to $1$. So, we can have an isometric dilation in this case also as shown below in Theorem \ref{thm:301}. However, the dilation space is just a normed linear space then.

\end{enumerate}

\begin{thm}\label{thm:301}
Let $(T_1, T_2)$ be a commuting pair of contractions on a normed linear space $\X$ such that the function $A_{T_i}: \X \to [0,\infty)$ given by
\[
 A_{T_i}(x) = \left(\|x\|^2 - \|T_ix\|^2\right)^{\frac{1}{2}}, \quad x\in \X \, ,
\]
defines a norm on $\X$ for $i=1,2$. Let $\X_i=\left(\X, A_{T_i}\right)$ for $i=1,2$. Then $(T_1,T_2)$ dilates to a commuting pair of isometries on a normed linear space if any one of the following conditions holds:
\begin{enumerate}
\item[(i)] there is a unitary $S: \X_1 \oplus_2 \X_2 \to \X_1 \oplus_2 \X_2$ such that
\[
  S(T_2x, x) = (x, T_1x), \quad x\in \X ;
\]
\item[(ii)] there are unitarily equivalent subspaces $\Y_1$, $\Y_2$ of $\X_1 \oplus_2 \X_2$ such that 
\[
  \widehat{M}_1 \oplus_2 \Y_1 = \X_1 \oplus_2 \X_2= \widehat{M}_2\oplus_2 \Y_2,
\]
where
$
  \widehat{M}_1 =\{(T_2x,x): x\in \X\}$, and $\widehat{M}_2=\{(x,T_1x): x\in \X\}$.
\end{enumerate}

\end{thm}

\begin{proof}
Suppose condition-$(i)$ holds. If $\X$ is a Banach space and $\|T_i\|< 1$ for $i=1, 2$, then this is just Theorem \ref{thm:202}. If $\|T_i\|=1$ for at least one $i$, then it is evident from the above discussion that the space $\X \oplus_2 \ell_2(\X_1 \oplus_2 \X_2)$ is not a Banach space. However, the maps $V_1$ and $V_2$ as in \eqref{eq:207} remain isometries. The rest of the proof works similarly as of the proof of Theorem \ref{thm:202}.

\smallskip

If $\X$ is a normed linear space but not a Banach space, then the space $\X \oplus_2 \ell_2(\X_1 \oplus_2 \X_2)$ is also just a normed linear space and is not a Banach space. The rest of the argument is exactly the same as in the previous case.

\smallskip 

Now, suppose condition-$(ii)$ holds. Then, by the last part of Lemma \ref{lem:201}, we have the existence of a unitary $S$ satisfying the condition-$(i)$ of the present theorem. Consequently, the desired dilation can be obtained by the previous part (i.e., part-$(i)$) of the present theorem.
\end{proof}

We conclude this paper here. Parrott showed in his paper \cite{SP} the existence of a commuting triple of strict contractions $(T_1,T_2,T_3)$ on a Hilbert space that does not dilate to any commuting triple of isometries. One can expect that there exists $r \in (0,1)$ such that any commuting Hilbert space triple $(T_1,T_2,T_3)$ with $\|T_i\|\leq r$ for each $i$ must dilate to a commuting triple of isometries. However, finding such an $r$ or a sharp bound on $r$ seems very challenging at this point.

\end{document}